\documentclass[12pt]{amsart}
\usepackage{amsmath}
\usepackage{amssymb}
\usepackage{amsfonts}
\usepackage{amsthm}
\usepackage{verbatim}
\usepackage{amscd}
\usepackage{cite}
\usepackage{leftidx}
\usepackage{enumerate}
\usepackage{txfonts}
\usepackage{manfnt}
\usepackage{amscd}
\usepackage[mathscr]{eucal}
\usepackage{hyperref}
\usepackage{palatino}

\def\tD{{\tilde D}}

\def\inv{^{-1}}

\DeclareMathOperator{\del}{\partial}

\def\refp #1.{(\ref{#1})}

\def\sbr #1.{^{[#1]}}
\def\sfl #1.{^{\lfloor #1\rfloor}}

\def\what{\widehat}
\def\inv{^{-1}}
\def\?{{\bf{??}}}

\def\dgla{differential graded Lie algebra\ }

\def\HH{\mathbb H}

\def\C{\mathbb C}
\def\P{\mathbb P}

\def\sym{\text{\rm Sym} }

\def\Spec{\text{\rm Spec} }

\def\pf{\mathrm {Pf}}

\def\O{\mathcal O}

\def\g{\mathfrak g}

\def\m{\mathfrak m}

\def\1/2{\frac{1}{2}}

\def\I{\mathcal{ I}}

\def\d{\mathrm{d}}

\def\2{{[2]}}

\def\nl{\newline}

\def\<{\langle}
\def\>{\rangle}

\def\2{{[2]}}

\def\Def{\mathrm{Def}}

\def\scl #1.{^{\lceil#1\rceil}}
\def\spr #1.{^{(#1)}}
\def\sbc #1.{^{\{#1\}}}

\def\subpr#1.{_{(#1)}}

\def\beq{\begin{equation*}}
\def\eeq{\end{equation*}}

\newcommand{\sing}{{\mathrm{sing}}}

\newcommand{\llog}[1]{\langle\log {#1} \rangle}
\def\g3{{\Gamma\spr 3.}}

\def\Def{\mathrm {Def}}

\newcommand{\eqspl}[2]{
\begin{equation}\label{#1}
\begin{split}
#2\end{split}\end{equation}}
\newcommand{\eqsp}[1]{\begin{equation*}
\begin{split}#1\end{split}\end{equation*}}

\newcommand{\beginalphaenum}{
\begin{enumerate}\renewcommand{\labelenumi}{ }
\item \begin{enumerate}
}

\def\eex{\end{rm}\end{example}}

\def\g{\mathfrak g}

\newtheorem{thm}{Theorem} 

\newtheorem*{thm*}{Theorem}
\newtheorem*{prop*}{Proposition}
\newtheorem{cor}[thm]{Corollary}
\newtheorem*{cor*}{Corollary}

\newtheorem{lem}[thm]{Lemma}
\newtheorem{lem*}{Lemma}

\newtheorem*{claim*}{Claim}
\newtheorem{prop}[thm]{Proposition}

\theoremstyle{remark}

\newtheorem{rem}[thm]{Remark}
\newtheorem{crit-rem}[thm]{Critical remark}
\newtheorem{remarks}[thm]{Remarks}

\newtheorem{example}[thm]{Example}
\newtheorem*{example*}{Example}

\newtheorem*{defn*}{Definition}

\begin{document} 
\title{Deformations of log-Lagrangian submanifolds of Poisson manifolds}
\author 
{Ziv Ran}



\thanks{arxiv.org/1311.2656}
\date {\today}


\address {\nl UC Math Dept. \nl
Big Springs Road Surge Facility
\nl
Riverside CA 92521 US\nl 
ziv.ran @  ucr.edu\nl
\url{http://math.ucr.edu/~ziv/}
}


\begin{abstract}
We consider Lagrangian-like submanifolds in certain even-dimensional
'symplectic-like' Poisson manifolds. We show, under a suitable
transversality hypothesis,  that the pair consisting of the ambient Poisson manifold
and the submanifold has unobstructed deformations,
and that the deformations automatically preserve the Lagrangian-like property.

\end{abstract}
\maketitle
The study of holomorphic    
Lagrangian submanifolds of compact holomorphic symplectic manifolds
and their deformation theory is well
established (see e.g. \cite{gross-cy-book} and references therein). 
Voisin \cite{voisin-lagrangian} proved 
that pairs $(X, Y)$ consisting of a compact K\"ahlerian symplectic
manifold $X$ and a Lagrangian submanifold $Y$ have
unobstructed deformations, i.e. an appropriate
Kuranishi space is polydisc; and under these deformations $Y$ stays
Lagrangian. See also \cite{hitchin-lagrangian}. \par Recently, refining some results
of Goto \cite{goto} and Hitchin \cite{hitchin-poisson},
we studied in  \cite{qsymplectic} certain even-dimensional  compact K\"ahlerian
Poisson
manifolds called {pseudo-symplectic}, from the point of view that they are analogous to
symplectic manifolds. A Poisson structure $\Pi\in\mathrm H^0(X, \bigwedge^2 T_X)$ 
on a complex manifold $X$ of dimension
$2n$ is said to be
 \emph{pseudo-symplectic}\footnote{As a referee suggests, 'log-symplectic' might be a better term;
 we also thank this referee for suggesting the term 'log-Lagrangian' used below.}
  if it is almost everywhere nondegenerate. When that is the case, $\Pi$ will
 degenerate along an anticanonical divisor $D=[\Pi^n]$ called the \emph{Pfaffian}
 of $\Pi$. We introduced a condition on $\Pi$ called \emph{P-normality}, which says
 that $D$ has normal crossings and is smooth wherever $\Pi$ has corank exactly 2.
Roughly speaking, the P-normality condition means that $(X,\Pi)$ is locally
 a product of Poisson manifolds of the form
 \emph{ (smooth surface, smooth anticanonical divisor).}
 We showed under this condition that $(X,\Pi)$ has, in a strong sense, unobstructed deformations.\par
 Here we consider an analogue of Lagrangian submanifolds in the Poisson
 setting. An $n$-dimensional closed
 submanifold $Y$ in a $2n$-dimensional
pseudo-symplectic manifold $(X, \Pi)$ is said 
to be \emph{log-Lagrangian} if 
 for all $y\in Y$, the conormal space $\check{N}_{Y,y}$
is an isotroptic subspace for $\Pi$.
Note that the log-Lagrangian condition implies that for all $y\in Y\setminus D$,
i.e. all points $y\in Y$- if any- where $\Pi$ is nondegenerate, $T_{Y, y}$ is a maximal
isotropic subspace for the symplectic form $\Phi_y=\Pi\inv_y$.\par
As an immediate example (see also Example \ref{hilbert-example} below), if $S$
is a smooth surface endowed with an effective anticanonical divisor $C=[\Pi], \Pi\in H^0(\O_S(-K))$,
there is a natural Poisson structure $\Pi\sbr r.$ on the Hilbert scheme $S\sbr r.$ (see \cite{qsymplectic}).
If  $B$ is any smooth curve on $S$, then $B\spr r.\subset S\sbr r .$ is a log-Lagrangian
submanifold.\par
 In this paper we will
restrict our attention to log-Lagrangians satisfying a transversality condition.
Specifically, if $(X\ \Pi)$ is a pseudo-symplectic Poisson manifold
with Pfaffian $D$,
a log-Lagrangian  $Y$  is said to be 
\emph{transverse}  (to $D$ or to $\Pi$ if specification be needed) if 
$X$ is P-normal in a neighborhood of $Y$, and
locally at every point $y\in Y\cap D$, the intersection $\bar D=D\cap Y$ is 
a normal-crossing divisor on $Y$ (this transversality condition 
holds vacuously if $D\cap Y$ is empty).
When $Y$ is transverse, we have that
$\check{N}_{Y,y}$ is isotropic for all $y\in Y$.
In the Hilbert scheme example above, $B\spr r.$ is transverse if $B$ is transverse to $C$;
in fact, $D$ is smooth in a neighborhood of $D\cap B\spr r.$.\par
Given $(X,\Pi)$ pseudo-symplectic Poisson and $Y\subset X$ log-Lagrangian,
one may consider deformations of the triple $(X, \Pi, Y)$ with or without the condition
that $Y$  stay log-Lagrangian. 
Let us call such deformations \emph{Poisson-Lagrange} or \emph{Poisson} respectively.
One may also consider deformations of $Y$ in $X$, with $(X, \Pi)$ fixed, again with or
without the Lagrangian condition; these will be named \emph{Lagrange} or \emph{Hilbert}
respectively. 
In fact, we will prove that the deformation
spaces with or without the Lagrangian condition are identical and smooth; in particular,
the log-Lagrangian property is 'sticky', impossible to move away from. 
Thus, our  main purpose is to prove
\begin{thm*}
Let $Y$ be a compact transverse log-Lagrangian submanifold of a
K\"ahlerian Poisson
manifold $(X, \Pi)$ that is P-normal along $Y$. Then\par
(i) $Y$ has unobstructed Hilbert and Lagrange deformations in $X$, and these deformation
spaces are identical;
\par
(ii) if moreover  $\Pi$ is P-normal  everywhere and 
$X$ is compact,  the triple $(X,\Pi,Y)$ has unobstructed Poisson-Lagrange and Poisson
deformations,
and these deformation spaces are identical.
\end{thm*}
This result 
includes the special case  where $Y$ is empty, which is the main result of \cite{qsymplectic},
as well as the special case where $\Pi$ is symplectic (so $D$ is empty: Voisin's theorem
\cite {voisin-lagrangian}). In fact, we will prove a more precise result (see Theorem
\ref{XYD-thm} below).\par
From a different viewpoint, some results on deformations of submanifolds of Poisson manifolds
were obtained by Baranovsky, Ginzburg et al. \cite{baranovsky}.
\vskip 1in
\section{Pr\'ecis of deformation theory}
\subsection{Lie atom, SELA} A (dg) \emph{Lie atom} \cite{atom} is the mapping cone associated to a Lie module
homomorphism
\[i:\g\to\mathfrak h\]
from a (differential graded) Lie algebra $\g$, endowed with the left regular representation,
to a $\g$-module. 
A Lie atom  is the special case corresponding to a 1-simplex 
of  a \emph{semi-cosimplcial object}
in the category of (differential graded) Lie algebras (called
in \cite{sela} semi-simplicial Lie algebras or selas for short).
In \cite{atom} we developed the deformation theory of Lie atoms, and
this was extended in \cite{sela} to the case of a sela.
In any case, the deformation theory connection is based on interpreting bracket-induced maps
on cohomology as obstructions.
For this paper, an important example of a Lie atom is the shift $N[-1]$
where $N$ is the normal bundle of a submanifold $Y\subset X$ of a compact complex manifold,
realized as the mapping cone of the inclusion
\[i:T_{X/Y}\to T_X\]
where $T_{X/Y}$ denotes the sheaf of holomorphic vector fields tangent to $Y$ along $Y$.
Then the deformation theory of $N[-1]$ is the deformation theory of $Y$ as submanifold of $X$
(i.e. the local structure of 
the Hilbert scheme or Douady space), which goes back to Kodaira-Spencer and Grothendieck
(see e.g. \cite{sernesi}, \S 3.4.4). Its starting point is the classical obstruction map
$\sym^2(H^0(N))\to H^1(N)$.
An example of a sela is the sela $\mathcal T^\bullet_X$ associated to an algebraic scheme $X/\C$, 
whose deformation theory is that of $X$ as $\C$-scheme.\par
\subsection{Jacobi-Bernoulli complex}
Briefly put, associated to a sela $\g$ satisfying suitable hypotheses (e.g. finiteness, vanishing 
of $\HH^0$, where
the latter
corresponds to having  trivial automorphisms),
 there is a comultiplicative 
complex $J^\bullet(\g)$ in \emph{strictly negative} degrees, 
called a Jacobi-Bernoulli complex, initially defined for Lie atoms  in \cite{atom}, \S 1.2, 
such that for the $(-n)$-th truncation $J_n(\g)$, we have that
\[R_n(\g):=\C\oplus \HH^0(J_n(\g))^*\] 
is a (commutative associative) Artinian local
ring 
classifying $n$-th order deformations
of $\g$. 
Essentially, the $(-i)$-th term of $J$ is $\bigwedge\limits^i(\g)$; however when $\g$ is not a dgla
(i.e. when the simplicial complex in question is not reduced to a point), some of the differentials in the complex
must be twisted by Bernoulli numbers (hence the name).
See e.g.\cite{ atom}, Thm. 3.3 (which is sufficient for this paper) and \cite{sela}, Thm 1.3 and Thm. 6.1.
\par
\subsection{$T^1$- lifting} Many results on unobstructed deformations, including those in this paper,
rely on a connection with Hodge theory
which has appeared before in a number of guises. Perhaps the most transparent one is
the '$T^1$-lifting criterion' of \cite{ran-torsion-or-negative}. In a nutshell,
this runs as follows. Suppose we have 'suitable' isomorphisms
\[\phi_i:T^i\to H^i, i=1,2\]
from a deformation group (resp. obstruction group) to a Hodge cohomology group.
Let $v\in T^1$. Then the obstruction $o(v)\in T^2$ to lifting $v$ to higher order
coincides via $\phi$ with the obstruction $\langle v, \phi_1(v)\rangle$ to deforming the cohomology
class $\phi(v)$  in the direction $v$, i.e.
\[\phi_2(o(v))=\langle v, \phi_1(v)\rangle\]
(in the cases where this is used, both sides are computed in  terms of Lie brackets, which ultimately 
turn out to be the same bracket).
However, the RHS above vanishes as $\phi(v)$ is
essentially topological in character (this requires the deformation corresponding to $v$ to be locally
trivial, which holds in our case). Therefore $o(v)=0$. Thus, the deformation space corresponding to $T^1$ is unobstructed.\par
This heuristic reasoning is made precise in various references, including
\cite{ran-torsion-or-negative}, \cite{qsymplectic}, \cite{voisin-lagrangian}.
\vskip 1in
\section{The normal dg atom}
See \cite{ciccoli} or \cite{dufour-zung} for Poisson basics.\par
\subsection{Normal DG atom and Lagrangian deformations}
A \emph{PL triple} $(X, \Pi, Y)$ by definition consists of
a a log-Lagrangian submanifold $Y$ of a pseudo-symplectic manifold
$(X, \Pi)$. 
Fixing a PL triple $(X, \Pi, Y)$, 
our purpose in this section is to describe a Lie-theoretic object (a dg Lie atom)
which controls log-Lagrange deformations of $Y$, i.e. 
deformations of $Y$ in $X$, fixing $X$, preserving the log-Lagrangian property.
In the next subsection, this will be extended to deformations of the entire PL triple.\par
Let $N$ denote the normal bundle of $Y$ and let
\[N^\bullet=\bigoplus\limits_{i=1}^n\bigwedge^iN\] be the exterior algebra on $N$.
 As discussed above, it is classical (and independent of any Poisson or Lagrangian properties)
that $N$ has a bracket structure that controls Hilbert deformations of $Y$ in $X$.
From the Lie-atomic viewpoint, this is developed in detail in \cite{atom} (see especially \S 3.2, Example 1.1.4.D).
In this section we will develop the log-Lagrangian analogue, where the main new feature
is the differential on $N^\bullet$. Thus we will
prove (compare \cite{baranovsky}):
\begin{thm}
Notations as above,\par (i) $N^\bullet$
admits the structure of differential graded Lie atom;\par
(ii) if $Y$ is compact, $N^\bullet$-deformations coincide with Hilbert-Lagrange deformations of $Y$ in $X$,
and the projection $N^\bullet\to N$ corresponds to the forgetful map from Hilbert-Lagrange to Hilbert
deformations.
\end{thm}
This result is not new. The existence of the differential on $N^\bullet$ was certainly known
to Baranovsky et al. \cite{baranovsky}, as was, in some form, the relationship of $N^\bullet$
to log-Lagrangian deformations.
\begin{proof}[Proof of Theorem]
(i) As mentioned above, the Lie atom structure of $N$ holds generally without any Poisson or Lagrange
conditions, and
was discussed in \cite{atom}. It is deduced from viewing it as the 
(shifted) mapping cone of the inclusion of Lie algebra sheaves
\[T_{X/Y}\to T_X\]
where $T_{X/Y}$ denotes the torsion-free sheaf of vector fields on $X$ tangent to $Y$. This structure
induces a graded Lie atom structure on $N^\bullet$, deduced from the mapping cone of
\[T_{X/Y}T^\bullet_X\to T^\bullet_X\]
where $T^\bullet_X$ is the Schouten graded Lie algebra and $T_{X/Y}T^\bullet_X$ the exterior ideal
generated by $T_{X/Y}$, which is easily seen to be a graded Lie subalgebra, though
not a Lie ideal.\par
The Poisson structure $\Pi$ and log-Lagrangian condition on $Y$
enter into the differential (on $T^\bullet_X$, hence on $N^\bullet$).
To see that the differential of $T^\bullet_X$, i.e. $[.,\Pi]$, descends to $N^\bullet$ is suffices to show that
the subalgebra 
$T_{X/Y}T^\bullet_X$ is closed under $[.,\Pi]$, and by elementary properties of the Schouten bracket
it suffices to prove closedness of $T_{X/Y}$, i.e. to show that
\[[T_{X/Y}, \Pi]\subset T_{X/Y}T_X.\] 
To show this note that the latter subsheaf of $\bigwedge\limits^2 T_X$
consists precisely of the local bivectors that pair to zero with $df_1\wedge df_2$ for
all $f_1, f_2\in \I_Y$.
Then, let $v$ be a local vector field 
on $X$ tangent to $Y$ (i.e. preserving the ideal sheaf $\I_Y$),
and let $f_1, f_2$ be local functions in $\I_Y$. Then by a standard formula of Lichnerowicz, we have
\[\langle df_1\wedge df_2, [v,\Pi]\rangle=\pm v(\{f_1, f_2\})\pm 
\langle (dv(f_1)\wedge df_2-dv(f_2)\wedge df_1),\Pi\rangle.
\]
This vanishes on $Y$ by the Lagrangian condition, which shows that 
$[v, \Pi]\in T_{X/Y}T_X\subset \bigwedge^2T_X$.\par
Assertion (ii) follows from the stronger result, Theorem \ref{triple-thm} below.

\end{proof}
\subsection{Deformations of PL triples}
We will denote the \dgla  $T_{X/Y}T^\bullet_X $ seen above by $T^\bullet_X\{Y\}$.
Thus in degree $i$,  $T^i_X\{Y\}$ is the subsheaf of $\bigwedge\limits^i T_X$ 
locally generated by sections of the form
\[u\wedge v_1\wedge...\wedge v_{i-1}\]
with $u$ a section of $T_{X/Y}$ and the $v_j$ sections of $T_X$.
By a \emph{Poisson-Lagrange} deformation of a PL  triple $(X, \Pi, Y)$ as above
we mean a triple $(\tilde X, \tilde \Pi, \tilde Y)$ so that $(\tilde X, \tilde \Pi)$
is a Poisson deformation of $(X, \Pi)$, $(\tilde X, \tilde Y)$ is a deformation
of $(X, Y)$, and $\tilde Y$ is log-Lagrangian (isotropic) with respect to $\tilde\Pi$. Dropping 
the last condition leads to (plain) \emph{Poisson} deformations of $(X, \Pi, Y)$.
\begin{thm}\label{triple-thm}
Assume $X$ is compact. Then the deformation theory of $T^\bullet_X\{Y\}$ coincides with the 
Poisson-Lagrange deformation theory of
the triple $(X,\Pi,Y)$ .
\end{thm}
\begin{proof}
Given the theory of \S 1, what's 
being asserted is that given a local Artinian algebra $R$, Poisson-Lagrange
 deformations
of $(X,\Pi, Y)$ are in bijective correspondence
with comultiplicative elements  of the Jacobi- Bernoulli cohomology group
$\HH^0(J(T^\bullet_X\{Y\}, R)$. In proving this assertion, we may assume
 the corresponding assertions for the differential graded Lie algebras
$T^\bullet_X$ and $T_{X/Y}$ with $R$ coefficients, as well as 
for $T^\bullet_X\{Y\}$ with coefficients in $R_1, \dim_\C(R_1)<\dim_\C(R)$,
to be true. The compactness assumption on $Y$ ensures that the groups
in question are all finite-dimensional.\par
Thus let $R_1=R/(\eta)$ where $\eta$ is in the socle $\mathrm{Ann}_R(\m_R)$,
and  suppose given a deformation diagram
\eqspl{}{
\begin{matrix}
\tilde Y&&\subset&&(\tilde X, \tilde \Pi)\\
&\searrow&&\swarrow&\\
&&\Spec(R)&&
\end{matrix}
} so that $(\tilde X, \tilde \Pi)$ is a Poisson deformation, $(\tilde Y\subset \tilde X)$ is a flat 
deformation, and so that the pullback over $R_1$ is a Poisson-Lagrange
deformation. The obstruction to $\tilde Y$ being log-Lagrangian over $R$
is the Poisson bracket
\[\{.,.\}:\I_{\tilde Y}\times \I_{\tilde Y}\to\O_{\tilde Y}\]
and by our assumption that the assertion above holds for $R_1$ in place of $R$, this map factors through
a pairing
\eqspl{obstruction}{\I_Y\times \I_Y\to \O_Y.}
Note that the obstruction to $\tilde Y$ being log-Lagrangian 
is of a local nature, so in analyzing it we may
choose compatible local coordinates on $X$ and $Y$ and assume 
that the deformations $\tilde X$ and $\tilde Y$
are, separately and not necessarily compatibly, trivial: i.e.
\[\tilde X\simeq X\times \Spec(R),\ \tilde Y\simeq Y\times\Spec(R).\]
Then the pairwise
deformation $\tilde Y\to \tilde X$ corresponds to a map
\eqsp{v:\I_Y\to \m_R\otimes\O_Y, v\in H^0(N)\otimes \m_R\\
\I_{\tilde Y}=\{f+v(f): f\in I_Y.\}
} Then in these terms the obstruction \eqref{obstruction} is given by
\[(f_1, f_2)\mapsto \{v(f_1),f_2\}-\{v(f_2),f_1\}-v(\{f_1, f_2\}).\]
(by our assumptions this is in $\eta\O_Y\subset \m_R\O_Y)$.
On the other hand, in terms of the Poisson differential $[.,\Pi]$,
this is exactly $\langle [v, \Pi], df_1\wedge df_2\rangle$, QED.
\end{proof}
\vskip 1in
\section{Unobstructed deformations}
\subsection{Set-up and statement}
We will keep the notations of the previous section. 
Thus, $(X,\Pi)$ is a
holomorphic Poisson manifold, not necessarily compact,
and $Y$ is a compact log-Lagrangian submanifold.
We now add the further hypotheses:\par
(i) $X$ is K\"ahlerian and P-normal along $Y$, with Paffian divisor $D$;\par 
(ii) $Y$ is {transverse}.\par
The transversality assumption means that $\bar D=D\cap Y$
has normal crossings. This is equivalent to the following condition:
let $D\sbr i.$ denote the $i$-fold locus of $D$, which
is locally a union of smooth branches of codimension $i$ in $X$. 
Then every branch of $D\sbr i.$
is transverse to $Y$, for all $i>0$. Note that this condition is strictly stronger
than the condition that every local branch of $D$ itself is transverse to $Y$.\par
We recall that Poisson deformations (resp. Poisson-Lagrange deformations)
of $(X, \Pi, Y)$ are deformations where $X$ deforms holomorphically,
$\Pi$ deforms as Poisson structure, and $Y$ deforms as arbitrary
(resp. log-Lagrangian) submanifold. The Poisson deformation space of
$(X, \Pi, Y)$ coincides with the fibre product 
\eqsp{
\begin{matrix}
\Def(X, \Pi)&&&&\Def(X, Y)\\
&\searrow&&\swarrow&\\
&&\Def(X).&&
\end{matrix}
}
Here it is convenient to interpret these and similar deformation spaces as
formal completions of suitable Kuranishi spaces.
We denote
by  $\mathrm{Def_{loc.\ trival}}(X, D, Y)$ the space of deformations of the
triple $(X, D, Y)$ where $D$ deforms locally trivially. This space
corresponds to the dgla $T_X\{Y\}\llog{D}$.
\begin{thm}\label{XYD-thm}Notations as above,  $Y$ has unobstructed Hilbert and Hilbert-Lagrange deformations
in $X$, and these coincide;
the space of first-order deformations of $Y$ in $X$ is canonically isomorphic to $H^0(Y, \Omega^1_Y\llog{\bar D})$.
\par
Furthermore, if $X$ is compact and P-normal, 
then the following assertions hold.\par 
(i) The triple $(X,\Pi, Y)$ has unobstructed Poison-Lagrange and Poisson deformations
and these deformations
coincide and induce locally trivial deformations on $D$.\par
(ii) The deformation space $\mathrm{Def_{loc.\ trival}}(X, D, Y)$ is  unobstructed.
\par
(iii) There is a deformation space of quadruples $(X, \Pi, D, Y)$ that maps
smoothly to $\Def(X, \Pi, Y)$ and to $\mathrm{Def_{loc.\ trival}}(X, D, Y)$
\end{thm}
\begin{rem}
A map $D_1\to D_2$ of deformation spaces is smooth iff any infinitesimal deformation
associated to $D_2$, parametrized by an Artinian (finite-dimensional) $\C$-algebra $R$,
lifts to a deformation associated to $D_1$ and parametrized by $R$.
\end{rem}
As in \cite{qsymplectic}, we deduce directly from the Theorem:
\begin{cor}
Assumptions as above with $X$ compact. Given a deformation $(\tilde X, \tilde Y)$ of $(X, Y)$, the Poisson structure $\Pi$ extends
to $(\tilde X, \tilde Y)$ iff $D$ extends locally trivially to $(\tilde X, \tilde Y)$.
\end{cor}
\subsection{Proof}
\begin{proof}[Proof of Theorem]
Let $\bar D$ be the restriction of the Pfaffian divisor $D$ on $Y$. By
our hypotheses, both $D$ and $\bar D$ have normal crossings. 
Henceforth, we will denote by $\Omega^\bullet$ various de Rham complexes truncated to
\emph{strictly positive} degrees (i.e. omitting the zeroth term  $\Omega^0=\O$).
Denote by
$\Omega^\bullet_X\{Y\}$ the kernel of the pullback map $\Omega^\bullet_X\to\Omega^\bullet_Y$.
Thus, $\Omega^1_X\{Y\}$ is locally generated by $\I_Y\Omega^1_X$ 
together with elements of the form $df, f\in \I_Y$;
and $\Omega^\bullet _X\{Y\}$ is generated by $\Omega^1_X\{Y\}$ as \emph{exterior ideal}, i.e.
\[\Omega^i_X\{Y\}=\Omega^1_X\{Y\}\wedge \Omega^{i-1}_X, i>1.\]
The basic formula
\[[df, dg]=d\{f,g\}\]
shows that the log-Lagrangian hypothesis implies
 $\Omega^1_X\{Y\}$ is a Lie subalgebra of $\Omega^1_X$ under Poisson-Lie bracket.
Then, using the usual derivation properties of brackets,
 it is not hard to check that $\Omega^\bullet_X\{Y\}$ has the structure of \dgla
so that the inclusion into the Lie-Poisson algebra $\Omega^\bullet_X$ is  a Lie subalgebra.
This turns the cokernel $\Omega^\bullet_Y$ into a differntial graded Lie atom.
Likewise, for the log differentials $\Omega^\bullet_X\{Y\}\llog{D}$,
a subalgebra of $\Omega^\bullet_X\llog{D}$ with
cokernel atom $\Omega^\bullet_Y\llog{\bar D}$. Now recall
the homomorphism $\bigwedge^\bullet\Pi^\#$ already used in \cite{qsymplectic}. It
yields a map of short exact sequences
\eqspl{duality}{\begin{matrix}
0\to&\Omega^\bullet_X\{Y\}\llog{D}\to &\Omega^\bullet_X\llog{D}\to&\Omega^\bullet_Y\llog{\bar D}\to & 0\\
&\downarrow&\downarrow&\downarrow&\\
0\to&T^\bullet_X\{Y\}\llog{D}\to & T^\bullet_X\llog{D}\to  &\ \ \ N^\bullet\to & 0.
\end{matrix}
}
[Regarding the (possibly surprising)
cokernel of the rightmost vertical map, note that the transversality of $Y$ and $D$ implies that the 
equations of the branches of $D$
at each point of $Y\cap D$ may be assumed to be part of a local coordinate system on $Y$,
while the passage from $T^\bullet_X$ to $T^\bullet_X\{Y\}$ affects only coordinates normal to $Y$;
hence the cokenel of $T^\bullet_X\{Y\}\llog{D}\to  T^\bullet_X\llog{D}$ is the same as
that of $T^\bullet_X\{Y\}\to  T^\bullet_X$, i.e. $N^\bullet$. Heuristically, it is clear by transversality that
motions of $Y$ in $X$ extend to (locally trivial) motions of $(Y, \bar D)=(Y, Y\cap D)$ so the appropriate
notion of 'log-normal complex' or "$N^\bullet\llog{\bar D}$" is just $N^\bullet$.]\par
The first two vertical maps are dgla homomorphisms, hence the right vertical arrow is a
Lie atom homomorphism. In any event, a local computation  in \cite{qsymplectic} shows
that the middle vertical arrow is bijective, and the same computation
also shows that the left vertical arrow is bijective.\par
We will now prove (i). The argument is the same as in the proof of the main theorem in \cite{qsymplectic}: Delgne's $E_1$ degeneration
theorem implies $E_1$-degeneration for $\Omega^\bullet_X\llog{D}$ and $\Omega^\bullet_Y\llog{\bar D} $,
hence for $\Omega^\bullet_X\{Y\}\llog{D}$.
Consequently,  by a variant of the T1-lifting criterion reviewed in \S 1, the bracket pairing
induces the trivial pairing on cohomology for the algebra $T^\bullet_X\{Y\}\llog{D}$,
hence this algebra has
unobstructed deformations.
Indeed in this case the vanishing
of obstructions is almost immediate from the fact that the exterior derivative operator $d$ induces
the zero map on $\HH^\bullet(\Omega^\bullet_X\{Y\}\llog{D}$, plus the standard formula for
Poisson-Lie bracket
\[[\omega_1, \omega_2]=\langle d\omega_1, \Pi^\sharp(\omega_2)\rangle-\langle d\omega_2,
\Pi^\sharp(\omega_1)-d(\Pi^\sharp(\omega_1\wedge \omega_2))\]
where $\Pi^\sharp$ denotes the duality operator (essentially interior multiplication by $\Pi$)
which yields a null homotopy for the bracket-induced map
\[\sym^2(\Omega^\bullet_X\{Y\}\llog{D})\to\Omega^\bullet_X\{Y\}\llog{D}.\]\par 
 Then we see as in \cite{qsymplectic},  \S 3.2,   that the inclusion
\[T^\bullet_X\{Y\}\llog{D}\to T^\bullet_X\{Y\}\] is a direct summand projection, so that
$(X, \Pi, Y)$ has unobstructed Poisson-Lagrange deformations.
Specifically, we define a complex $K^\bullet\{Y\}$ by
\eqspl{}{
K^0\{Y\}=T_X\{Y\};\\
K^1\{Y\}=T^2_X\{Y\}\oplus N^0_D;\\
K^i\{Y\}=T^{i+1}_X\{Y\}\oplus T^{i-1}_X\otimes N^0_D, \ i\geq 2
} where $N^0_D\subset N_D=\O_D(D)$ is the image of $T_X\to N_D$, which also coincides
with the image of $T_X\{Y\}\to N_D$ by transversality (NB: 
this explains why we don't need to define something
like $N^0_D\{Y\}$- it would be the same as $N^0_D$). 
The maps are as in \cite{qsymplectic}, \S 3.2,
as is the proof that the map $T^\bullet\{Y\}\llog{D}\to K^\bullet\{Y\}$ 
is a quasi-isomorphism. As there, we have a map
$\phi:T^\bullet\{Y\}\to K^\bullet\{Y\}$ which yields a quasi-splitting 
to the inclusion $T^\bullet\{Y\}\llog{D}\to T^\bullet\{Y\}$.
The components of $\phi$ going from $T^i_X\{Y\}$ to $T^i_X\{Y\}$
are the identity.
 The component of $\phi$ going from
$T^2_X\{Y\}$  to $N^0_D$ is given by 
\[u\mapsto nu\wedge\Pi^{n-1};\]
then as usual  this is extended to $T^{i+1}_X\{Y\}\to T^{i-1}_X\otimes N^0_D, i>1$ as an exterior
derivation, i.e. by
\[v_1\wedge...v_{i+1}\mapsto \sum\limits_{a<b} 
(-1)^{a+b}v_1\wedge...\wedge \hat{v_a}\wedge...\wedge\hat{v_b}\wedge...
\wedge v_{i+1}\otimes\phi(v_a\wedge v_b).\]
 \par

The fact that Poisson-Lagrange and Poisson deformations coincide
is a consequence of surjectivity of the edge map
\[\HH^\bullet(T^\bullet_X\{Y\}\llog{D}\to H^\bullet (T_{X/Y}\llog{D})\]
which in turn is a special case of the $E_1$ degeneration. Indeed a Poisson
deformation of $(X, \Pi, Y)$ induces a locally trivial deformation of $(X, D, Y)$,
and the latter deformations are controlled exacly by the dgla $T_{X/Y}\llog{D}$.
This proves assertion (i) in the Theorem.
\par
As for assertions (ii) and (iii), consider the deformation space 
of quadruples $(X, \Pi, D, Y)$ that is associated to the dgla $T^\bullet_X\{Y\}\llog{D}$.
As we have seen, this is smooth as follows from $E_1$ degeneration for the
latter complex. 
The fact that the inclusion
$T^\bullet_X\{Y\}\llog{D}\to T^\bullet_X\{Y\}$ admits a left quasi-inverse, i.e. a left inverse
in the derived category, implies its surjectivity in cohomology, whence smoothness of
the induced map on deformation spaces. The smoothness of $\mathrm{Def_{loc. trivial}}(X, D, Y)$,
and of the map to it from the deformation space of $T^\bullet_X\{Y\}\llog{D}$, 
follows from surjectivity
on cohomology of the edge map  
\[\Omega^\bullet_X\{Y\}\llog{D}\to  \Omega^1_X\{Y\}\llog{D}\]
which is a consequence of
Deligne's result on
$E_1$ degeneration for  $\Omega^\bullet_X\{Y\}\llog{D}$ \cite{deligne-hodge-iii}.
This establishes assertion (ii) and (iii).
\par
Finally, the smoothness and equality of Hilbert and Hilbert-Lagrange deformations of $Y$ assuming only
compacness of $Y$ is a consequence of the diagram \eqref{duality} (bijectivity of
the right vetical arrow, which uses only P-normality of $X$ along $Y$ and 
transversality), plus Delgne's $E_1$ degeneration for
$\Omega^\bullet_Y\llog{\bar D}$, which implies the vanishing of obstructions for the dg Lie atoms
$N^\bullet\simeq \Omega^\bullet_Y\llog{\bar D}$ and $N\simeq \Omega^1_Y\llog{\bar D}$ 
and surjectivity of the edge map \[\HH^0(N^\bullet)\to 
H^0(N).\]

\end{proof}
\subsection{An example}
\begin{example}\label{hilbert-example}
Let $S$ be a smooth compact surface and $C\subset S$ a smooth anticanonical divisor.
Then $C$ corresponds to a Poisson structure $\Pi$ on S. As shown in \cite{qsymplectic}, $\Pi$
induces a Poisson structure $\Pi\sbr r.$ on the degree-$r$ Hilbert scheme $S\sbr r.$.
This structure is not
P-normal. Now let $r=2$. Then $\Pi\sbr 2.$ still is not P-normal: its Paffian divisor is the locus of schemes
having nonempty intersection with $C$, and has Whitney umbrella-
type singularities; but $\Pi\sbr 2.$ lifts to a P-normal  Poisson structure $\Pi_2$ on the blowup $f:X_2\to S\sbr 2.$ 
of $S\sbr 2.$ in $C\spr 2.$,
the locus of schemes contained in $C$. We saw in \cite{qsymplectic}
that the deformation space of $(S\sbr 2.,\Pi\sbr 2.)$
is isomorphic to that of $(X_2, \Pi_2)$, hence is unobstructed.
Briefly, the argument that the
respective deformation spaces of 
$(S\sbr 2.,\Pi\sbr 2.)$ and $(X_2, \Pi_2)$ are isomorphic runs as follows.
A deformation of $(S\sbr 2., \Pi\sbr 2.)$ induces a locally trivial deformation of $\pf(\Pi\sbr 2.)$,
 hence a deformation of $C\spr 2.$ as the singular locus of $\pf(\Pi\sbr 2.)$, hence also a deformation 
 of $(X, \Pi_2)$. Conversely, it is well known
 (e.g. \cite{horikawa} ) that a deformation of a blowup of a manifold along a smooth
 submanifold induces a 
  a deformation of the blowdown morphism. In particular, a deformation of $X_2$
induces a deformation of $S\sbr 2.$, therefore deformations of $(X, \Pi^2)$ induce deformations of 
 $(S\sbr 2., \Pi\sbr 2.)$.
 \par
Now let $B\subset S$ be a smooth curve transverse
to $C$. Then $B\spr 2.\subset S$ is transverse to $C\spr 2.$, 
and
$B^2=f\inv(B\spr 2.)\simeq B\spr 2.$ is transverse to $D=\pf(\Pi^2)=f\inv(\pf(\Pi\sbr 2.))$ and is
 log-Largrangian for $\Pi^2$. As noted above, the respective deformation spaces of 
 $(S\sbr 2., \Pi\sbr 2.)$ and $(X, \Pi^2)$ are naturally isomorphic, hence so are those of
 $(S \sbr 2., \Pi\sbr 2., B\spr 2.)$ and $(X_2, \Pi^2, B^2)$. 
 \par

 Therefore, $(S, \sbr 2., \Pi\sbr 2., B\spr 2.)$
 has unobstructed (Poisson, or equivalently Poisson-Lagrange) deformations.\par
For $r>2$, $\Pi\sbr r.$ is still P-normal along $B\spr r.$, because locally a subscheme
 of $B$ can have at most a length-1 intersection with $C$; therefore 
 by the Theorem, $B\spr r.$ has unobstructed
(Hilbert or Langrange: they are the same) deformations in $S\sbr r.$.\par
 \emph{Conjecturally},  as noted in \cite{qsymplectic},  the rest of this example also extends to the
 case $r>2$, because it is conjectured that the following
 process yields a P-normal blowup of $(S\sbr r., \Pi\sbr r.)$: 
blowing up the $r$-fold locus of the Pfaffian of $\Pi\sbr r.$, then 
blowing up the proper transform of the $(r-1)$st fold locus, etc.
This conjecture is known to hold over the open set in $S\sbr r.$ consisting of curvilinear schemes 
(whose complement  has codimension $>2$). It appears to hold more generally over
the set of \emph{locally monomial} subschemes (defined locally by a monomial ideal); and this
would seem to imply the general case since every subscheme is a 
deformation of a locally monomial one. The details of the case $r>2$ are yet to be written down. 
\end{example}
\begin{rem}
Christian Lehn \cite{lehn-lagrangian} has generalized the Voisin theorem to normal-crossing 
Lagrangian \emph{subvarieties} $Y$. The analogous statement in the Poisson setting remains open 
(recall that in our treatment we are always assuming a Lagrangian $Y$- unlike the Pfaffian divisor $D$-
to be smooth).
\end{rem}
\vskip 2in
\vfill\eject
\bibliographystyle{amsplain}
\bibliography{../mybib}
\end{document}